\documentclass{amsart}

\usepackage{amsmath}
\usepackage{amssymb}
\usepackage{amsthm}
\usepackage{enumerate}
\usepackage{amsfonts}
\usepackage{setspace}
\usepackage{color}

\newtheorem{thm}{Theorem}[section]
\newtheorem{lem}[thm]{Lemma}
\newtheorem{cor}[thm]{Corollary}
\newtheorem{prop}[thm]{Proposition}
\newtheorem{ob}[thm]{Observation}

\theoremstyle{definition}

\theoremstyle{remark}
\newtheorem{remark}[thm]{Remark}

\numberwithin{equation}{section}

\newcommand{\M}{\mathcal{M}}

\newcommand{\J}{\mathcal{J}}

\newcommand{\N}{\mathbb{N}}
\newcommand{\CC}{\widehat{\mathbb{C}}}
\newcommand{\R}{\mathbb{R}}
\newcommand{\Q}{\mathbb{Q}}
\newcommand{\C}{\mathbb{C}}
\newcommand{\D}{\mathbb{D}}
\newcommand{\Z}{\mathbb{Z}}
\newcommand{\closure}{\overline}

\begin{document}

\title{A remark on Zagier's observation of the Mandelbrot set}

\author{Hirokazu Shimauchi}
\address{Graduate School of Information Sciences, \endgraf Tohoku University, Sendai, 980-8579, Japan.}
\email{shimauchi@ims.is.tohoku.ac.jp, hirokazu.shimauchi@gmail.com}

\date{April 1, 2013 and, in revised form, March 24, 2014.}

\begin{abstract}
We investigate the arithmetic properties of the coefficients for the normalized conformal mapping of the exterior of the Multibrot set.
In this paper, an estimate of the prime factor of the denominator of these coefficients is given.
Bielefeld, Fisher and Haeseler \cite{laurent} presented Zagier's observation for the growth of the denominator of the coefficients for the Mandelbrot set, and Yamashita \cite{yamashita} verified it.
Our study takes into consideration both Zagier's observation and Yamashita's estimate.
\end{abstract}

\maketitle

\section{Introduction}
Let $\C$ be the complex plane, $\CC$ the Riemann sphere, $\D$ the open unit disk, $\D^{*}$ the exterior of the closed unit disk, and $\Z^+$ the set of non-negative integers.
We denote the $n$-th iteration of polynomial $P$ by $P^{\circ n}$, which is defined inductively by $P^{\circ (n+1)}=P \circ P^{\circ n}$ with $P^{\circ 0}(z)=z$. The \textit{Julia set} $\J_{P}$ of $P$ is the set of all $z \in \CC$ such that $\{ P^{\circ n} \}_{n=0}^{\infty}$ is not normal in any neighborhood of $z.$
We consider the complex dynamical systems of the polynomials given by $P_{d,c}(z):=z^d+c$ on $\CC$, where $c \in \C$ is a parameter and $d \in \N \setminus \{1\}$ is fixed.
The \textit{Multibrot set} $\M_d$ of degree $d$ is the set of all parameters $c \in \C$ for which $\J_{P_{d,c}}$ is connected.
The original \textit{Mandelbrot set} is derived in the case $d=2$.
It is known that $\M_d$ is compact and is contained in the closed disk of radius $2^{1/(d-1)}$ with center $0$ (see e.g. {\cite[pp. 20-22]{taniguchi}} and \cite{sym}).

Constructing a conformal homeomorphism $\Phi_2:\CC \setminus \M_2 \rightarrow \D^*$ that satisfies $ \Phi_2(z)/z \rightarrow 1 $ as $ z \rightarrow \infty$, Douady and Hubbard \cite{iteration} proved the connectedness of the Mandelbrot set.
The statement can be generalized to $d \geq 2$ in the same way.
For more details, see \cite{laurent} for $d=2$ and \cite{yamashita,ICFIDCAA} for the generalized statement.

\begin{thm}[{\cite[Theorem 2]{iteration}}]
There exists a conformal homeomorphism $\Phi_d: \CC \setminus \M_d \rightarrow \D^*$ that satisfies $\lim_{z \rightarrow \infty} \Phi_d (z)/z = 1 .$
\end{thm}

We define the map $\Psi_d: \D^* \rightarrow \CC \setminus \M_d$ by $\Psi_d (z)= \Phi_d^{-1}(z)$, where $\Phi_d^{-1}$ is the inverse map of $\Phi_d$.
Let $b_{d,m}$ be the coefficients of the Laurent expansion of $\Psi_d$ as $$\Psi_d (z)=z+\sum_{m=0}^{\infty} b_{d,m} z^{-m}.$$
If $\Psi_d$ extends continuously to the unit circle, then, according to Carath\'{e}odory's continuity theorem, the Mandelbrot set is locally connected (see e.g. {\cite[p.154]{taniguchi}}).
There is an important conjecture which states that the Mandelbrot set is locally connected; actually, this conjecture includes the conjecture for the density of hyperbolic dynamics in the quadratic family (see {\cite[p.5]{orsay}}).
It should be noted that the convergence of the Laurent series on the unit circle implies that $\M_d$ is locally connected.
With this in mind, we investigate the arithmetic properties of the above-stated coefficients $b_{d,m}$.

\section{Zagier's observation of the denominator}
Jungreis \cite{jungreis} presented a method to compute the coefficients $b_{2,m}$ of $\Psi_2$.
Levin \cite{levin2} showed an algorithm for computing the coefficients of the function $\Phi_2'/ \Phi_2$.
Using Levin's algorithm, $b_{2,m}$ can be computed.

Jungreis's method can be generalized to $d \in \N \setminus \{1\}$ similarly (see \cite{yamp,yamashita,ICFIDCAA}).
We can make a program for this procedure and derive the exact value of $b_{d,m}$, because the following proposition holds.
We call a real number $x$ \textit{$d$-adic rational} if there exist $k \in \Z$ and $n \in \Z^+$ such that $x=k d^{-n}$.

\begin{prop}[\cite{jungreis}]\label{pro:jun}
The coefficients $b_{d,m}$ are $d$-adic rational numbers.
\end{prop}

Moreover, Jungreis \cite{jungreis} showed that some of these coefficients are $0$, which we call \textit{zero-coefficients}.
Several detailed investigations on $b_{d,m}$ are presented in \cite{laurent, coef, area, levin} for the case $d=2$, and in \cite{levin2,sym,yamp,yamashita} for general $d$.
In particular, Bielefeld, Fisher and Haeseler presented Zagier's observations in \cite[pp. 32-33]{laurent}.
One of the observations is the necessary and sufficient condition for the coefficients $b_{2,m}$ to be zero.
Some information about the zero-coefficients can be found in \cite{laurent, levin, levin2, sym, yamp, yamashita}.
However, the necessary and sufficient condition for the zero-coefficients is still unknown.
The other observations are related to the growth of the denominator of $b_{2,m}$.
In this paper, we focus on the denominator of $b_{d,m}$.

We prepare the notation about the $p$-adic valuation.
For any prime number $p$ and every non-zero rational number $x$, there exists a unique integer $v$ such that $x=p^v r/q$ with integers $r$ and $q$ that are not divisible by $p$.
The $p$-adic valuation $\nu_p:\Q \setminus \{0\} \rightarrow \Z$ is defined as $\nu_p(x)=v.$
We extend $\nu_p$ to $\Q$ as follows:
\begin{eqnarray}
 \nonumber
 \nu_p (x)=
  \begin{cases}
    v & \text{ for } x \in \Q \setminus \{0\},\\
    +\infty & \text{ for } x = 0.\\
  \end{cases}
\end{eqnarray}
For any real number $x$, we set the floor function $\left \lfloor x \right \rfloor := \max\{ m \in \Z : m \leq x  \}.$
We would like to note that the floor function and the $p$-adic valuation have the following properties (see e.g. {\cite[pp.69-72,  and pp.102-114]{gra}}).

\begin{lem}\label{pro:prime}
Let $x,y \in \R$ and $p$ be a prime number. 
The floor function and the $p$-adic valuation satisfy the following:
\begin{eqnarray}
&& \label{eq:ineq0} \left \lfloor x \right \rfloor + m \leq \left \lfloor x + m \right \rfloor \text{ for } m \in \Z^+.\\
&& \label{eq:ineq1} \left \lfloor x \right \rfloor + \left \lfloor y \right \rfloor \leq \left \lfloor x + y \right \rfloor.\\
&& \label{eq:ineq2} \left \lfloor \frac{\left \lfloor x \right \rfloor}{m} \right \rfloor  = \left \lfloor \frac{x}{m} \right \rfloor \text{ for } m \in \Z^+.\\
&& \label{eq:ineq4} \nu_p (mn) = \nu_p (m) + \nu_p (n) \text{ for } m,n \in \Z.\\
&& \label{eq:ineq5}\nu_p (m/n) = \nu_p (m) - \nu_p (n) \text{ for } m \in \Z, n \in \Z \setminus \{0\}.\\
&& \label{eq:ineq6} \nu_p (m+n) \geq \min\{ \nu_p (m), \nu_p (n) \} \text{ for } m, n \in \Z.\\
&& \label{eq:ineq3} \nu_p (m!) = \displaystyle \sum_{l=1}^{\infty} \left \lfloor \frac{m}{ p^l } \right \rfloor \text{ for } m \in \Z^+.
\end{eqnarray}
\end{lem} 

Zagier's observation for the growth of the denominator is as follows:

\begin{ob}[{\cite[Observation (ii)]{laurent}}]\label{ob:z}
For $0 \leq m \leq 1000$, 
$$ -\nu_{2} (b_{2,m}) \leq \nu_{2} ((2m+2)!) $$
holds. Furthermore, equality holds if and only if $m=0$ or $m$ is odd.
\end{ob}

We would like to note the following inequality presented by Ewing and Schober in \cite{area}.
\begin{thm}[{\cite[Theorem 1]{area}}]
For any $m \in \Z^+$, $ -\nu_{2} (b_{2,m}) \leq 2m+1 $ holds.
\end{thm}

Furthermore Levin \cite{levin2} showed that equality holds in Zagier's observation if $m$ is odd.
\begin{thm}[{\cite[Theorem on p. 3520]{levin2}}]
If $m \in \Z^+$ is an odd integer, then $ -\nu_{2} (b_{2,m}) = \nu_{2} ((2m+2)!)  $ holds.
\end{thm}

Yamashita \cite{yamashita} defined the order of integers with respect to $d \in \N \setminus \{1\}$  and gave an estimate of the growth of the denominator of $b_{d,m}$ (see \cite[Theorem 4.30]{yamashita}).
However, the theorem was incorrect.
We learnt from Yamashita that there was a mistake in the calculation.
However, if we restrict $d$ to be a prime number, his result is valid.
Using the $p$-adic valuation, his statement can be presented in the following way.

\begin{thm}[{\cite[Theorem 4.30]{yamashita}}]\label{thm:yame}
Let $m \in \Z^+$ and $p$ be a prime number.
If $(p-1) \mid (m+1)$,
$$ -\nu_{p} (b_{p,m}) \leq \left \lfloor \frac{\nu_{p} ((pm+p)!) }{p-1} \right \rfloor$$
holds.
Otherwise, $-\nu_{p} (b_{p,m})=-\infty$ holds. 
Furthermore, if $(p-1) \mid (m+1)$, equality holds precisely when $m=p-2$ or $p \nmid m$.
\end{thm}

Hence Observation \ref{ob:z} holds for general $m$.
In this paper, we give an estimate for the prime factor of the denominator of $b_{d,m}$.
Our estimate includes Observation \ref{ob:z} and is equal to Yamashita's estimate if $d$ is a prime number.
Actually the general statement of Observation \ref{ob:z} and Theorem \ref{thm:yame} are given as a corollary of our main results.
Furthermore we obtain another corollary for the growth of the denominator of $b_{d,m}$.

\section{Main Result}
First, we introduce Ewing and Schober's coefficients formula for general $d$.
The result for $d = 2$ is given in \cite{coef};
however, the argument can be generalized to $d \geq 2$ similarly.
For more details, see {\cite[Lemma 4.18]{yamashita}}.

\begin{lem}[{\cite[Theorem 1]{coef}}]\label{lem:int}
For $n \in \Z^+$, let $1 \leq m \leq d^{n+1}-3$ and $R>0$ be sufficiently large.
Then,
\begin{eqnarray} \label{eq:int}
b_{d,m}= - \frac{1}{2\pi m i } \int_{|z|=R} P_{d,z}^{\circ n}(z)^{m/d^n} \, \mathrm{d}z .
\end{eqnarray}
\end{lem}

Using the recursion $P_{d,c}^{\circ n}(z)=(P_{d,c}^{\circ (n-1)} (z))^d+c$ for (\ref{eq:int}), we obtain the following formula for coefficients in the same way (see {\cite[Corollary 4.20]{yamashita}}).
Let $C_j\left(a\right)$ denote the general binomial coefficient, i.e., $$C_j\left(a\right)=\displaystyle\frac{a(a-1)(a-2) \cdots (a-j+1)}{j(j-1)(j-2) \cdots (1)}$$ for any $a \in \R$ and $j \in \N$ with $C_0\left(a\right)=1$.

\begin{cor}[{\cite[Corollary]{area}}]\label{th:comb}
Let $n \in \N$ and $1 \leq m \leq d^{n+1}-3$. Then,
\begin{eqnarray}\label{eq:comb}
\nonumber
b_{d,m}=& - \displaystyle\frac{1}{m} \sum  & C_{j_1} \left( \frac{m}{d^n} \right) C_{j_2} \left( \frac{m}{d^{n-1}}-dj_1 \right) C_{j_3} \left( \frac{m}{d^{n-2}}-d^2j_1-dj_2 \right) \\
\nonumber
&& \cdots C_{j_n} \left( \frac{m}{d}-d^{n-1}j_1-d^{n-2}j_2 -\dots -dj_{n-1} \right),
\end{eqnarray}
where the sum is over all non-negative indices $j_1, \dots, j_n$ such that $(d^n-1)j_1+(d^{n-1}-1)j_2+(d^{n-2}-1)j_3 + \cdots +(d-1)j_n=m+1$.
\end{cor}

Levin \cite{levin2}, Lau and Schleicher \cite{sym}, and Yamashita \cite{yamp} proved the following lemma independently.
Furthermore, using Corollary \ref{eq:comb}, Yamashita presented an alternative short proof of Lemma \ref{lem:yamp} in \cite{yamashita}.

\begin{lem}[{\cite[Theorem on p. 3515]{levin2} },{\cite[p. 47]{sym} },{\cite[Theorem 2]{yamp}}]\label{lem:yamp}
Let $d \geq 3$ and $m \in \Z^+$. If $ (d-1) \nmid (m+1)$, then $b_{d,m}=0$ holds.
\end{lem}

Our main result is given as below.

\begin{thm} \label{th:main}
Let $d \in \N \setminus \{1\}$, $m \in \Z^+$ with $(d-1) \mid (m+1)$ and set $a=(m+1)/(d-1)$.
Assume that $d$ is factorized as $d=p_1^{t_1} p_2^{t_2} \dots p_s^{t_s}$, where $s \in \N, t_1,t_2,\dots, t_s \in \N $ and $p_1, p_2, \dots, p_s$ are distinct prime numbers.
Then,
$$ -\nu_{p_i} (b_{d,m}) \leq \nu_{p_i} (a!) + t_i a$$
holds for any $p_i \in \{p_1, p_2, \dots, p_s\}$.
Furthermore, equality holds precisely when $m=d-2$ or $p_i \nmid m$.
\end{thm}

\begin{remark}
We disregard the case $(d-1) \nmid (m+1)$ because $-\nu_{p_i} (b_{d,m})=-\infty$ holds for any $p_i$ by Lemma \ref{lem:yamp}.
There are some zero-coefficients which are not included in Lemma \ref{lem:yamp} (see \cite{yamp}).
\end{remark}

\begin{proof}
Let $m \in \Z^+$ and assume that $d$ is factoraized as $d=p_1^{t_1} p_2^{t_2} \dots p_s^{t_s}$ where $s \in \N, t_1, t_2, \dots, t_s \in \N $ and $p_1, p_2, \dots ,p_s$ are distinct prime numbers.
From direct calculation, $b_{2,0}=-1/2$ (see \cite{jungreis,laurent}). 
Furthermore $b_{d,0}=0$ for $d \geq 3$ by Lemma \ref{lem:yamp}.
Hence, the statement holds for $m=0$.

For fixed $p_i \in \{ p_1, p_2, \dots, p_s \}$, we take $n \in \N$ such that $m \leq d^{n+1} -3$ and $\nu_{p_i}(m) < nt_i$.
We define the set $\mathbb{J}$ as:
\begin{eqnarray}
\nonumber
\mathbb{J} &=& \bigl\{ ( j_1,j_2, \dots, j_n ) \in (\Z^{+})^n :\\
\nonumber
&& (d^n - 1)j_1 + (d^{n-1} - 1)j_2 + \dots + (d-1)j_n =m+1 \bigl\}.
\end{eqnarray}
Let $k \in \{1,2, \dots, n\}$, $(j_1,j_2,\cdots,j_n) \in \mathbb{J} $ and set 
$$\alpha = m/d^{n-k+1} - d^{k-1}j_1- d^{k-2}j_2 - \cdots - d j_{k-1}$$
and
$$\beta = d^{n-k+1} \alpha = m - d^{n}j_1- d^{n-1}j_2 - \cdots - d^{n-k+2} j_{k-1}.$$
Then we have
\begin{eqnarray}
\nonumber
C_{j_k}( \alpha ) &=& \displaystyle \frac{ \alpha ( \alpha -1)( \alpha -2) \cdots ( \alpha - (j_k-1))}{j_k !}\\
&=& \frac{ \beta ( \beta - d^{n-k+1})( \beta - 2 d^{n-k+1} ) \cdots ( \beta - (j_k-1) d^{n-k+1} )}{ d^{j_k(n-k+1)}  j_k!}. \label{eq:cj}
\end{eqnarray}
The denominator of the fractional expression (\ref{eq:cj}) satisfies the following equality:
\begin{eqnarray}
\nu_{p_i} \left( d^{j_k(n-k+1)}  j_k! \right) &\overset{(\ref{eq:ineq4})}{=}& \nu_{p_i} (j_k !) +j_k t_i (n-k+1). \label{eq:add1}
\end{eqnarray}
On the other hand we have
\begin{eqnarray}
\nonumber
&& \nu_{p_i} \left( \beta - (l-1) d^{n-k+1} \right) \\
\nonumber
&=& \nu_{p_i} \left( m - d^{n}j_1- d^{n-1}j_2 - \cdots - d^{n-k+2} j_{k-1} - (l-1) d^{n-k+1} \right)\\
\nonumber
&=& \nu_{p_i} \left( m - d^{n-k+1} \left( d^{k-1}j_1- d^{k-2}j_2 - \cdots - d j_{k-1} - (l-1) \right) \right)\\
\nonumber
&\overset{(\ref{eq:ineq6})}{\geq}& \min\{ \nu_{p_i} (m), t_i (n-k+1) \}
\end{eqnarray}
for all $l \in \{ 1, 2, \cdots, j_k \}$.
Hence the numerator of the fractional expression (\ref{eq:cj}) satisfies the following inequality:
\begin{eqnarray}
\nonumber
&& \nu_{p_i} \left( \beta ( \beta - d^{n-k+1})( \beta - 2 d^{n-k+1} ) \cdots ( \beta - (j_k-1) d^{n-k+1} ) \right) \\
&\overset{(\ref{eq:ineq4})}{\geq}& j_k \min\{ \nu_{p_i} (m), t_i (n-k+1) \}.\label{eq:add2}
\end{eqnarray}
Combining (\ref{eq:add1}) and (\ref{eq:add2}), we obtain
\begin{eqnarray}
\nonumber
-\nu_{p_i}( C_{j_k}( \alpha ) ) &\overset{(\ref{eq:ineq5})}{\leq}& \nu_{p_i} (j_k !) + j_k t_i (n-k+1) - j_k \min\{ \nu_{p_i} (m), t_i (n-k+1) \} \\
&=& \nu_{p_i} (j_k !) + j_k \max \{ t_i (n-k+1)- \nu_{p_i}(m),0  \}. \label{eq:cja}
\end{eqnarray}

By Corollary \ref{th:comb} and (\ref{eq:cja}), 
\begin{eqnarray}
\nonumber
-\nu_{p_i}( b_{d,m} )  &\leq& \nu_{p_i} (m) + \max_{( j_1,j_2, \dots, j_n) \in \mathbb{J}} \left  \{ \displaystyle \sum_{k=1}^n  \nu_{p_i} (j_k !) \right \}\\
\nonumber
&&+ \max_{( j_1,j_2, \dots, j_n) \in \mathbb{J}} \left \{ \displaystyle \sum_{k=1}^n   j_k \max \left  \{ t_i (n-k+1)- \nu_{p_i}(m),0  \right \} \right \}.
\end{eqnarray}

Due to 
\begin{eqnarray}
\label{eq:jcon} (d^n - 1)j_1 + (d^{n-1} - 1)j_2 + \dots + (d-1)j_n =m+1
\end{eqnarray}
for all $(j_1,  j_2, \dots, j_n) \in \mathbb{J}$ and (\ref{eq:ineq1}), we have
\begin{eqnarray}
\nonumber
\displaystyle \sum_{k=1}^n  \nu_{p_i} (j_k !) &\overset{(\ref{eq:ineq3})}{=}& \displaystyle \sum_{k=1}^{n} \sum_{l=1}^{\infty} \left \lfloor \frac{j_k }{ p_i^l } \right \rfloor\\
\nonumber
&\overset{ (\ref{eq:ineq1})}{\leq}& \sum_{l=1}^{\infty} \left \lfloor \frac{ j_1+ \dots + j_n }{ p_i^l} \right \rfloor\\
\nonumber
&\overset{(\ref{eq:jcon})}{\leq}& \sum_{l=1}^{\infty} \left \lfloor \frac{m+1} {(d-1) p_i^l} \right \rfloor\\
\nonumber
&\overset{(\ref{eq:ineq3})}{=}& \nu_{p_i} (a!)
\end{eqnarray}
for all $1 \leq k \leq n.$
Equality holds if $j_1=j_2= \dots = j_{n-1}=0, j_n = (m+1)/(d-1)=a$.
Conversely if equality holds, then $j_1=j_2= \dots = j_{n-1}=0, j_n = a$,
because $\displaystyle(j_1,j_2,\dots,j_n) \in \mathbb{J}$ satisfies $$(d^{n-1}+d^{n-2}+ \dots +1) j_1+ \dots +(d+1) j_{n-1} + j_n = \frac{m+1} {d-1}$$
and hence $$\max_{\{j_1,\dots,j_n) \in \mathbb{J}} \{ j_1+ \dots +j_n \} = \frac{m+1}{d-1}$$
is valid precisely when  $j_1=j_2= \dots = j_{n-1}=0, j_n = a$.

We consider the case of $p_i \nmid m$.
Since $\nu_{p_i}(m)=0$ for $p_i \nmid m$, we have
\begin{eqnarray}
\nonumber
&& \displaystyle \sum_{k=1}^n   j_k \max \left \{ t_i (n-k+1)- \nu_{p_i}(m),0  \right \}  \\
\nonumber
&=& \displaystyle \sum_{k=1}^n j_k t_i (n-k+1) \\
\nonumber
& \leq & \displaystyle \sum_{k=1}^{n} j_k t_i \frac{d^{n-k+1}-1}{d-1}  \\
\nonumber
&\overset{(\ref{eq:jcon})}{=}&t_i a
\label{eq:in1}
\end{eqnarray}
for all $( j_1,j_2, \dots, j_n) \in \mathbb{J}$.
Further, equality holds if and only if $j_1=j_2= \dots = j_{n-1}=0, j_n = a$.
We note that if $j_1=j_2= \dots = j_{n-1}=0, j_n = a$, then $$C_{0}(m/d^n)C_{0}(m/d^{n-1}) \cdots C_{a}(m/d)=C_{a}(m/d) \not= 0.$$
Considering the elements of the sum in the formula in Corollary \ref{th:comb},  we have
$b_{d,m} \not= 0$ and $ -\nu_{p_i} (b_{d,m}) = \nu_{p_i} (a!) + t_i a$.

We consider the other case $p_i \mid m$.
Since $\nu_{p_i}(m) < n t_i$, there exists $1 \leq n' \leq n$ and $( j_1,j_2, \dots, j_n) \in \mathbb{J}$ such that
\begin{eqnarray}
\nonumber
&& \max_{( j_1,j_2, \dots, j_n) \in \mathbb{J}} \left \{ \displaystyle \sum_{k=1}^n   j_k \max \left \{ t_i (n-k+1)- \nu_{p_i}(m),0  \right \} \right \}\\
\nonumber
&=& \displaystyle \sum_{k=1}^{n'}   j_k (t_i (n-k+1)- \nu_{p_i}(m))\\
\nonumber
&=& \displaystyle \sum_{k=1}^{n'}   j_k t_i (n-k+1) - \sum_{k=1}^{n'} j_k \nu_{p_i}(m).
\end{eqnarray}
On the one hand,
\begin{eqnarray}
\nonumber
\displaystyle \sum_{k=1}^{n'}   j_k t_i (n-k+1) &\leq& \displaystyle \sum_{k=1}^{n}   j_k t_i (n-k+1) \\
\nonumber
&\leq& \displaystyle \max_{( j_1,j_2, \dots, j_n) \in \mathbb{J}} \left \{ \sum_{k=1}^{n} j_k t_i \frac{d^{n-k+1}-1}{d-1} \right \}\\
\nonumber
&=& t_i a.
\end{eqnarray}
Thus, equality holds if and only if $j_1=j_2= \dots = j_{n-1}=0, j_n =a$.
On the other hand,
$$\sum_{k=1}^{n'} j_k \nu_{p_i}(m) \geq \nu_{p_i}(m).$$
The necessary and sufficient condition for equality is $j_1+j_2+ \dots + j_n =1$.
Hence, 
\begin{eqnarray}
\nonumber
\displaystyle \sum_{k=1}^{n'}   j_k t_i (n-k+1) - \sum_{k=1}^{n'} j_k \nu_{p_i}(m) \leq t_i a - \nu_{p_i} (m).
\end{eqnarray}
Equality holds if and only if $j_1=j_2= \dots = j_{n-1}=0, j_n =1$, which means that $m=d-2$.
Therefore, if $p_i  \mid m$,
$$ -\nu_{p_i} (b_{d,m}) \leq  \nu_{p_i} (m)  + \nu_{p_i} (a!) + t_i a - \nu_{p_i} (m) = \nu_{p_i} (a!) + t_i a,$$
and equality holds if and only if $m=d-2$.
\end{proof}

Considering Theorem \ref{th:main}, we can verify Observation \ref{ob:z} for the denominator of nonzero-coefficients and  Theorem \ref{thm:yame} by the same calculation of Yamashita \cite{yamashita}.

\begin{proof}[Proof of Observation \ref{ob:z} and Theorem \ref{thm:yame}]
Let $p$ be a prime number.
By Lemma \ref{lem:yamp} and that $b_{2,0}=-1/2$, we may assume $(p-1) \mid (m+1)$.
We obtain $$-\nu_{p}(b_{p,m}) \leq \frac{m+1}{p-1}+\nu_{p}\left(\frac{m+1}{p-1}!\right)$$ from Theorem \ref{th:main}.
By the direct calculation which is the same as Yamashita's calculation in \cite{yamashita},
\begin{eqnarray}
\nonumber
&& \frac{m+1}{p-1}+\nu_{p}\left(\frac{m+1}{p-1}!\right)\\
\nonumber
&\overset{ (\ref{eq:ineq3})}{=}& \frac{m+1}{p-1} + \displaystyle \sum_{l=1}^{\infty} \left \lfloor \frac{m+1}{(p-1)p^l} \right \rfloor\\
\nonumber
&\overset{ (\ref{eq:ineq2})}{=}& \frac{m+1}{p-1} + \displaystyle \left \lfloor \frac{1}{p-1} \sum_{l=1}^{\infty} \left \lfloor \frac{m+1}{p^l} \right \rfloor \right \rfloor\\
\nonumber
&\overset{ (\ref{eq:ineq0})}{=}&\displaystyle \left \lfloor \frac{1}{p-1} \sum_{l=0}^{\infty} \left \lfloor \frac{m+1}{p^l} \right \rfloor \right \rfloor\\
\nonumber
&\overset{ (\ref{eq:ineq3})}{=}& \displaystyle \left \lfloor \frac{\nu_{p} ((pm+p)!) }{p-1} \right \rfloor.
\end{eqnarray}
\end{proof}
Furthermore, we obtain the following estimate for the growth of the denominator of $b_{d,m}$.
For any real number $x$, we set the ceiling function $\left \lceil x \right \rceil := \min\{ m \in \Z : m \geq x  \}.$

\begin{cor}
Let $m \in \Z^+$, $d \in \N \setminus \{1\}$, $(d-1) \mid (m+1)$ and $a=(m+1)/(d-1)$.
Assume that $d$ is factorized as $p_1^{t_1} p_2^{t_2} \dots p_s^{t_s}$, where $s \in \N, t_1,t_2,\dots ,t_s \in \N $ and $p_1, p_2, \dots, p_s$ are distinct prime numbers.
Then,
$$b_{d,m} d^{x(m)} \in \Z,$$
where
$$x(m):= \displaystyle \max_{ 1 \leq i \leq s }  \left \lceil \frac{ \nu_{p_i} (a!) + t_i a }{t_i} \right \rceil .$$
\end{cor}

\section*{Acknowledgement}
The author would like to thank Professor Yohei Komori and Osamu Yamashita for their valuable advice and encouragement.
He further expresses his sincere gratitude to Professor Toshiyuki Sugawa for his helpful suggestions and constant encouragement.
Last but not least, the author is grateful to the referee for his/her careful reading of this manuscript and for his/her helpful comments.
\bibliographystyle{amsplain}

\end{document}